\documentclass[a4paper,12pt,reqno]{amsart}
\usepackage{amsmath,amsfonts,enumerate,amssymb,amsthm,color}
\usepackage{pgf,tikz}

\title{Cyclic groups are CI-groups for balanced configurations}
\author{Hiroki~Koike}
\address{H.~Koike, FAMNIT, University of Primorska, Glagolja\v ska 8, 
6000 Koper, Slovenia}
\email{hiroki.koike@upr.si} 
\author{Istv\'an~Kov\'acs}
\address{I.~Kov\'acs, IAM and FAMNIT, University of Primorska, Glagolja\v ska 8, 
6000 Koper, Slovenia}
\email{istvan.kovacs@upr.si}
\author{Dragan~Maru\v si\v c}
\address{D.~Maru\v si\v c, FAMNIT, University of Primorska, Glagolja\v ska 8, 
6000 Koper, Slovenia}
\email{dragan.marusic@upr.si}
\author{Mikhail~Muzychuk}
\address{M.~Muzychuk, Department of Computer Science and Mathematics, 
Netanya Academic College, University st. 1, 42365 Netanya, Israel}
\email{muzy@netanya.ac.il}

\thanks{
The first three authors were supported in part by ARRS, the Slovenian Research Agency (research program P1-0285); the second and third 
authors thank ARRS (research projects N1-0032, N1-0038, J1-5433 and J1-6720); and the third author also thanks 
ARRS (I0-0035), and H2020 Teaming InnoRenew CoE. 
The fourth author thanks the University of Primorska for hospitality.
}

\def\B{\mathcal{B}}

\def\Z{\mathbb{Z}}
\def\K{\mathcal{K}}
\DeclareMathOperator{\aut}{Aut}
\DeclareMathOperator{\fix}{\sf Fix}

\newtheorem{thm}{Theorem}[section]
\newtheorem{lem}[thm]{Lemma}

\newtheorem{prob}[thm]{Problem}
\theoremstyle{definition}

\theoremstyle{remark}
\newtheorem{rem}[thm]{Remark}

\newcommand{\comment}[1]{}

\newcommand{\sg}[1]{\langle{#1}\rangle}

\begin{document}

\maketitle

\begin{abstract}
In this paper it is shown that every finite cyclic group satisfies the CI-property for the class of balanced configurations. 
\end{abstract}

\section{Introduction}

Let $G$ be a finite group. Following \cite{P87}, a \emph{Cayley object} 
of $G$ is a relational structure with underlying set $G$ which is invariant under all 
right translations $x \mapsto xg, x,g \in G$. Given a class $\K$ of Cayley objects of $G,$ the group $G$ is said to satisfy the \emph{CI-property} for $\K$ if for any two isomorphic Cayley objects in $\K,$ there exists an isomorphism between them 
which is in the same time an automorphism of $G$. 
The question whether a given finite group $G$ satisfies the CI-property for a given class 
$\K$ of Cayley objects has been extensively studied under various choices of $G$ and 
$\K$. P\'alfy \cite{P87} characterized the finite groups satisfying the CI-property for all Cayley objects. It turns out that these are the groups $\Z_2 \times \Z_2,$ $\Z_4$ and the cyclic groups $\Z_n$ with $\gcd(n,\varphi(n)) = 1,$ where $\varphi$ is the Euler 
function. A great deal of research was devoted to the CI-property of finite groups for 
digraphs, all which date back to \'Ad\'am's famous conjecture \cite[Problem~2-10]{A67}.  This was eventually answered in  \cite{M95,M97,M97b} where it was proven that the cyclic group $\Z_n$ satisfies the CI-property for digraphs if and only if its order is of the form $n=k,$ $2k,$ or $4k,$ where $k$ is an odd square-free number. For further information on arbitrary finite groups satisfying the 
CI-property for graphs and digraphs, we refer to the survey \cite{L02}.    

In this paper we expose a nontrivial and well studied class of combinatorial objects - the class of balanced configurations  \cite{G09,PS13} -  for which, perhaps surprisingly, every finite cyclic group satisfies the CI-property.  

First, an \emph{incidence geometry} $(\Omega,\B)$ consists of a set $\Omega$ of $n$ points and a collection $\B  \subseteq 2^\Omega$ of $b$ lines (or blocks) such that 
$|B \cap B'| \le 1$ for any $B,B' \in \B,$ $B \ne B'$. Then $(\Omega,\B)$ is called a \emph{configuration} of type $(n_r,b_k)$ (\emph{combinatorial configuration} in the sense of \cite{G09}) if the following hold:
\begin{itemize}
\item $| \{ B \in \B : \omega \in B \} | = r$ for every $\omega \in \Omega;$ and 
\item $|B| = k$ for every $B \in \B$ with $k \ge 3$. 
\end{itemize}
The \emph{automorphism group} $\aut X$ of the configuration 
$X=(\Omega,\B)$ consists of all permutations $g \in {\rm Sym}(\Omega)$ which leaves the set $\B$ invariant, that is, $\B^g = \B$.

A configuration with $n = b$ (and therefore $r = k$) is called \emph{balanced,} 
or a \emph{$k$-configuration,} and its type is simply denoted by $(n_k)$. 
In this paper we prove the following theorem.

\begin{thm}\label{main}
Every finite cyclic group satisfies the CI-property for balanced configurations.
\end{thm}

\begin{rem}
It should be noted that some special cases of Theorem~\ref{main} have already been settled. 
Jungnickel \cite{J08} proved that finite abelian groups satisfy the 
CI-property for projective planes, while Koike et al.\ \cite{KKP14} proved 
Theorem~\ref{main} in the case when the cyclic group is of order a prime power or a product of two primes.
\end{rem}

In our proof of Theorem~\ref{main} we follow the approach used in \cite{M99}. 
Let $G$ be an arbitrary finite group, and $X$ be any Cayley object of $G$. We say that $X$ is a \emph{CI-object} if 
whenever $X$ is isomorphic to another Cayley object $Y$ of $G,$ there exists 
an isomorphism between $X$ and $Y$ which is at the same time an automorphism 
of $G$. Babai \cite{B77} gave the following group-theoretic reformulation of the 
CI-property: $X$ is a CI-object if and only if  the regular subgroups of $\aut X,$ which are isomorphic to $G$, form a single conjugacy class of subgroups in $\aut X$.

Following Li~\cite{L05}, we say that $G\leq{\rm Sym}(\Omega)$ is a \emph{c-group} if it contains a regular cyclic subgroup (or a full cycle). 
In \cite{M99}, the \emph{partial order $\prec$} on a set of c-subgroups  was defined as follows: for c-subgroups $A,B \leq {\rm Sym}(\Omega),$ let $A \prec B$ if 
$A \le B,$ and  for every regular cyclic subgroup $L \le B,$ $L^b \le A$ holds for 
some $b \in B$ (here and in what follows $L^b$ is the conjugate subgroup $b^{-1} L b$). 

We say that a c-group $G$ is \emph{$\prec$-minimal} if it is minimal with respect to the partial order $\prec$.  In this context Babai's criterion says 
that, a cyclic object $X$ with a regular cyclic subgroup 
$C \le \aut X$ is a CI-object if and only if $C$ is the unique $\prec$-minimal subgroup of $\aut X$. Since we can restrict ourselves to connected configurations only (see Section 2 for a detailed discussion),
Theorem~\ref{main} is a corollary of the following statement.

\begin{thm}\label{main2}
Suppose that $X=(\Omega,\B)$ is a connected balanced configuration and 
$C \le \aut X$ is a cyclic subgroup which is regular on $\Omega$. Let $C \le G \le \aut X$ be a $\prec$-minimal group. Then $G=C$. 
\end{thm}

 In Section~2 we 
give a number of lemmas which will be used in the proof of Theorem~\ref{main2} in Section~3.

\section{Preparatory lemmas}

Throughout the paper we keep the notation set in Theorem~\ref{main2}. 
Furthermore, we denote by $D_m,$ $SD_m$ and $Q_m$ the dihedral, semidihedral and 
generalized quaternion group of order $m,$ respectively (see, e.\ g., \cite{B08}). 
Clearly, the cyclic group $C$ in Theorem~\ref{main2} is of order $|C|=|\Omega|$. 
We let $n=|\Omega|$. Let $n_p$ denote the largest $p$-power dividing $n,$ where $p$ is a prime. For a divisor $d$ of $n,$ let $C_d$ be the subgroup of 
$C$ of order $d,$ and let $C_{p'}=C_{n/n_p}$. 
\medskip

Define the relation $R \subset \Omega \times \B$ by letting $(\omega,B) \in R$ 
if and only if $\omega \in B$. The corresponding 
undirected graph induced on $\Omega \cup \B$ with edge set $R \cup R^t$ is also known as the \emph{incidence graph} of $X$.  By $R^t$ we denote the relation 
transpose to $R,$ that is, $R^t=\{ (\alpha,\beta) : (\beta,\alpha) \in R \}$.
It is easily seen that $R$ is $\aut X$-invariant, 
and that there are no points $\alpha, \beta \in \Omega$ and lines 
$\gamma,\delta \in \B$ for which $\{\alpha,\beta\} \times \{\gamma,\delta\} \subseteq R$. In what follows we shall refer to the latter property of $R$ by saying that $R$ is \emph{$K_{2,2}$-free}.

For a subgroup $K \le \aut X,$ we denote by $\Omega/K$ the set 
of $K$-orbits of its  action on $\Omega,$ and by 
$\B/K$ the set of $K$-orbits of its action on $\B$. 

\begin{lem}\label{L0} Let $(\Omega,\B)$ be a balanced configuration 
of type $(n_k)$. If $C \leq \aut X$ is a cyclic subgroup which acts regularly on $\Omega$ then it acts regularly on $\B$ too.
\end{lem}

\begin{proof} We have to show that $C$ is transitive on $\B$. Assume  the contrary, that is, $C$ has more than one orbit on $\B$. Let $\B_1$ be one of them. Then $n_1:=|\B_1| < |\B|=n=|C|$ implying that a subgroup $C_{n/n_1}$ fixes each line $B \in \B_1$ setwise. Therefore each line from $\B_1$ is a union of $C_{n/n_1}$-orbits. In particular, if two lines of $\B_1$ intersect nontrivially, then they contain at least two common points. Since the relation $R$ is $K_{2,2}$-free, such intersection could never occur, and so $\B_1$ is a partition of 
$\Omega$. Hence $k n_1 = n,$ implying 
that each line $B\in \B_1$ is just a unique  $C_k$-orbit. Hence $\B_1 = \Omega/C_k$. Arguing in the same way for another $C$-orbit, 
$\B_2$ say, we conclude $\B_2 = \Omega/C_k = \B_1,$ a 
contradiction.
\end{proof}

Let us recall that the configuration $(\Omega,\B)$ in the above lemma 
is called \emph{connected} if its incidence graph is connected.
It follows from the above statement that the connected components of the incidence graph are permuted transitively by the cyclic group $C$. Thus the components are pairwise isomorphic, and, moreover, each of them is a balanced cyclic configuration by itself.  

Since the configurations are isomorphic if and only if their connected components are isomorphic, this observation reduces the general case to the one with a connected incidence graph. In the rest of the paper we shall assume that the relation $R\cup R^t$ is connected. 

\begin{lem}\label{L1} 
With the notation of Theorem~\ref{main2}, let $N$ be a normal subgroup of $G$ such that  $|\Omega/N|=|\B/N|$. 
Then $N$ acts faithfully on its orbits on both $\Omega$ 
and $\B$. In particular, $G$ acts faithfully on both $\Omega$ 
and $\B$.
\end{lem}

\begin{proof} Notice that the assumptions of Therorem~\ref{main2} and Lemma~\ref{L0} guarantee that $G$ is transitive on $\Omega$ and $\B$.

Pick an arbitrary pair $(\delta_1,\delta_2)\in R$ and consider the $N$-orbits
$\Delta_i=\delta_i^N, \, i=1,2$. Let $N_{\Delta_i}$ be the pointwise stabilizer of 
the set $\Delta_i$ in $N$. We claim that $N_{\Delta_1}=N_{\Delta_2}$. Indeed, if these subgroups are distinct, then we may assume without loss of generality that $N_{\Delta_2} \not\leq N_{\Delta_1}$. Hence $N_{\Delta_2}$ acts on $\Delta_{1}$ nontrivially. Since $N_{\Delta_i},\, i=1,2,$ are normal in $N$, the orbits of $N_{\Delta_2}$ on $\Delta_{1}$ have the same size greater than $1$. There exists $g_{2} \in N_{\Delta_{2}}$ such that $\delta_{1}^{g_2}\neq \delta_1$. 
Therefore $|R^t(\delta_2)\cap \Delta_1|\geq 2$. It follows from 
$$
|\Delta_1|=\frac{|\Omega|}{|\Omega/N|} = \frac{|\B|}{|\B/N|} =|\Delta_2|
$$
that $|R(\delta_1)\cap\Delta_2|=|R^t(\delta_2)\cap\Delta_1| \ge 2$. Let $\delta'_2 \neq \delta_2$ be a point from $R(\delta_1) \cap \Delta_2$. Then 
$$
(\delta_1,\delta_2), (\delta_1,\delta'_2)\in R\implies 
\{(\delta_1,\delta_2),(\delta_1,\delta'_2),
(\delta_1,\delta_2)^{g_2}, (\delta_1,\delta'_2)^{g_2}\}\subseteq R\implies
$$
$$ 
\{(\delta_1,\delta_2),(\delta_1,\delta'_2),
(\delta_1^{g_2},\delta_2), (\delta_1^{g_2},\delta'_2\} = 
\{\delta_1,\delta_1^{g_2}\}\times \{\delta_2,\delta'_2\}\subseteq R.
$$
This contradicts the fact that $R$ is $K_{2,2}$-free. Therefore, $N_{\delta_1^N}=N_{\delta_2^N}$ holds for any pair $(\delta_1,\delta_2)\in R$. By connectivity of $R \cup R^t$ we obtain that $N_{\Delta_1} = N_{\Delta_2}$ 
for any pair of $N$-orbits. Hence $N_{\Delta_i}$ is trivial.  
\end{proof}

\begin{lem}\label{L2}
With the notation of Theorem~\ref{main2}, for any $(\alpha,\beta) \in R$ and arbitrary nontrivial $F \leq C$ it holds that $\beta^F \not\subseteq \beta^{G_\alpha}$ and 
$\alpha^F \not\subseteq \alpha^{G_\beta}$. 
\end{lem}

\begin{proof} 
If $\beta^F \subseteq \beta^{G_\alpha}$, then $(\alpha,\beta^F) \subseteq R$ implies that $(\alpha,\beta^F)^F\subseteq R,$ hence $(\alpha^F,\beta^F) \subseteq R,$ 
a contradiction. The proof of $\alpha^F \not\subseteq \alpha^{G_\beta}$ is analogous.
\end{proof} 

\section{The proof of Theorem~\ref{main2}}

We keep the notation of Theorem~\ref{main2}. 
The key step will be to describe the Fitting subgroup $F(G)$.
Since $G$ is a $\prec$-minimal group, it is solvable (see \cite[Theorem~1.8]{M99}). Therefore, to get a description
of $F(G)$ we first analyse the maximal normal $p$-subgroups 
${\bf O}_p(G),$ where $p$ runs over the set of all prime divisors of $|G|$. 

\begin{lem}\label{L4} 
With the notation of Theorem~\ref{main2}, let $p$ be  a prime divisor of $|G|$ such that 
$\mathbf{O}_p(G) \not\leq C$. Then $p=2,$ $n_2 \ge 8,$ and 
\begin{enumerate}[(i)]
\item $\mathbf{O}_2(G) \cong D_{2n_2}$ or $SD_{2n_2},$  and 
$C_{n_2} < \mathbf{O}_2(G);$ or 
\item $\mathbf{O}_2(G) \cong D_{n_2}$ or $Q_{n_2}$,  
 $\mathbf{O}_2(G)C_{n_2} \cong D_{2n_2}$ or $SD_{2n_2},$ and 
$|C \cap \mathbf{O}_2(G)|=\frac{1}{2}n_2$. 
\end{enumerate}  
\end{lem}

\begin{proof} Fix a prime divisor $p$ of $|G|$ with nontrivial 
$\mathbf{O}_p(G)$. 
In what follows we abbreviate $P= \mathbf{O}_p(G)$. Let $p^e$ and $p^f$ be the sizes of $P$-orbits on $\Omega$ and $\B$ respectively. By the symmetric role of $\Omega$ and $\B,$ 
we may assume that $e \geq f$. If $f=0$, then $P$ acts trivially on $\B$ implying, by Lemma~\ref{L1}, that 
$\mathbf{O}_p(G)=\{1\}$. Thus we may assume that $f > 0$, which, in turn, implies $p\,|\, n$.

Let us consider the group $P C_{p^e}$. The orbits of $P C_{p^e}$ on $\Omega$ 
have cardinality $p^e$ and coincide with the orbits of $C_{p^e}$.  We claim that the same holds for the orbits of $P C_{p^e}$ on $\B$. 
Indeed, $C_{p^e}$ permutes the orbits of $P$ on $\B$, and the kernel of the action of $C_{p^e}$ on $\B/P$ is a subgroup of order $p^f$. Thus an orbit of $P C_{p^e}$ is a union of at most $p^{e-f}$  $P$-orbits. Therefore, the upper bound for the size of a $P C_{p^e}$-orbit is $p^e$. Finally, this yields that the size of a $P C_{p^e}$-orbit on $\B$ is equal to $p^e,$ and thus it must be equal to a $C_{p^e}$-orbit.   
Since $P \not\le C,$ it follows that $P C_{p^e} >  C_{p^e}$.  

We prove next that $p=2,$ $e \ge 3$ and 
\begin{equation}\label{PP1}
P C_{2^e} = \langle C_{2^e}, a \rangle = C_{2^e} \rtimes \langle a \rangle,
\end{equation}
where $a$ is an involution that acts on $C_{2^e}$ in such a way that for every $x \in C_{2^e},$ 
$x^a=x^{-1}$ or $x^a=x^{-1+2^{e-1}}$.

Since $C_{p^e} < P C_{p^e},$ we can choose a subgroup $C_{p^e} < N \le P C_{p^e}$ for which $N = \langle C_{p^e}, a' \rangle$ of order $p^{e+1},$ $a'$ is of order $p,$ $N \cong \Z_{p^e} \rtimes \Z_p,$ 
and $a'$ acts on $C_{p^e}$ by the rule $x^{a'} = x^{m}, x \in C_{p^e}$ where $m\in\{1,1+p^{e-1}\}$ if $p > 2$ and $m\in\{\pm 1,\pm 1 + p^{e-1}\}$ if $p=2$.
Notice that, in the case when $m=1,$ the semidirect product is equal to the direct product $N = C_{p^e} \times \langle  a' \rangle$.

Let us consider first the case when $m\in\{1,1+p^{e-1}\}$. 
For $m=1+p^{e-1}$ the group is also called the modular group of order 
$p^{e+1},$ and denoted by $M_{p^{e+1}}$.
In either case, the subgroups of $N$ of order 
$p$ generate an elementary abelian normal subgroup of order $p^2$. 
Let us denote this subgroup by $Q$.  Below we consider $N$ acting on 
$\Omega \cup \B$. 
Now, $C_{p^e}$ is transitive on every $N$-orbit, hence 
$N = N_\omega C_{p^e}$ for any $\omega \in \Omega \cup \B,$ 
$|N_\omega|=p$, and thus any orbit of $N_\omega$ has cardinality $1$ or $p$. Let us fix an arbitrary element 
$\omega \in \Omega \cup \B,$ and let  $\alpha \in \Omega \cup \B$ be such that $\alpha^{N_\omega} \neq \{\alpha\}$. 
Then $N_\alpha \neq N_\omega,$ hence $N_\alpha N_\omega = Q = N_\alpha C_p,$ and therefore, 
$$
\alpha^{N_\omega} = \alpha^{N_\alpha N_\omega} = \alpha^{N_\alpha C_p} =\alpha^{C_p}.
$$
We have shown that every $N_\omega$-orbit of size $p$  coincides with some $C_p$-orbit. Set $\Delta=\fix(N_\omega):=\{\alpha \in \Omega \cup \B : \alpha^{N_\omega}=\{\alpha\} \},$ let  
$\Delta_1=\Delta \cap \Omega$ and 
$\Delta_2=\Delta \cap \B$. Since, by Lemma~\ref{L1}, $G$ acts faithfully on each orbit, 
$\Delta_1$ is a proper subset of $\Omega$ and $\Delta_2$ is a proper subset of $\B$.
Notice that both sets $\Delta$ and $\Omega \setminus \Delta$ are 
nonempty.
To get a contradiction, we show that $R$ has no arc between $\Delta$ and $\Omega \setminus \Delta$.  Indeed, if  
$(\alpha,\beta) \in R$ for some $\alpha \in \Delta$ and $\beta \in \Omega \setminus \Delta$, then $(\alpha,\beta)^{N_\omega} \subseteq R$ implies that 
$(\alpha,\beta^{N_\omega}) \subseteq  R$. Since $N_\alpha=N_\omega$ and $N_\alpha \neq  N_\beta$, we conclude that 
$\beta^{N_\omega}=\beta^{N_\alpha}=\beta^{C_p},$ hence $\beta^{C_p} \subseteq \beta^{G_\alpha},$ contradicting Lemma~\ref{L2}. 
If  $(\alpha,\beta) \in R$ for some $\alpha \in \Omega \setminus \Delta$ and $\beta \in \Delta$, then $(\alpha,\beta)^{N_\omega} \subseteq R$ implies that 
$(\alpha^{N_\omega},\beta) \subseteq  R$. Since $N_\beta=N_\omega$ and $N_\alpha \neq  N_\beta$, we conclude that 
$\alpha^{N_\omega}=\alpha^{N_\beta}=\alpha^{C_p},$ hence $\alpha^{C_p} \subseteq \alpha^{G_\beta},$ contradicting Lemma~\ref{L2}. 

Therefore, $p=2$ and $m \in \{-1,-1 + 2^{e-1}\}$ with $e \ge 3$. 
We claim that $N = P C_{2^e}$. 
Let $L$ be the normalizer of $C_{2^e}$ in $PC_{2^e}$. It follows from 
$C_{2^e}\trianglelefteq N$ that  $N\leq L$.
It follows from $PC_{2^e} = C_{2^e} (PC_{2^e})_\omega$  and $C_{2^e} \triangleleft N\leq L \leq PC_{2^e}$ that $L = C_{2^e} L_\omega$.
By definition of $L$, we have $C_{2^e}\trianglelefteq L$. This gives us a homomorphism $\mu: L_\omega\rightarrow\Z_{2^e}^*$ defined by 
the rule 
$h^x = h^{\mu(x)}$ for $h \in C_{2^e}$ and $x \in L_\omega$. If there exists $x\in L_\omega\setminus\{1\}$ with $\mu(x)\equiv 1 ({\rm mod}\ 4)$, then a unique involution $b\in\sg{x}$ satisfies $\mu(b)\in\{1,1+2^{e-1}\}$. Now we can repeat the previous arguments for the subgroup $M=\sg{C_{2^e},b}$ which yields  $N=PC_{2^e}$ and \eqref{PP1}.

Thus we may assume that $\mu(x)\equiv -1 ({\rm mod}\ 4)$ for all $x\in L_\omega\setminus\{1\}$. But this is possible only if $|L_\omega|=2$.
Thus $L=N\cong\Z_{2^e}\rtimes\Z_2$. If $N\neq PC_{2^e}$, then there exists an element $g\in PC_{2^e}\setminus N$ which normalizes $N$. It follows from $N\cong\Z_{2^e}\rtimes\Z_2$ that all elements of $N\setminus C_{2^e}$ have order at most $4$. Since $e\geq 3$, the group $C_{2^e}$ is characteristic in $N$. Therefore $g$ normalizes $C_{2^e}$ implying $g\in L$, contrary to $L = N$. Thus $N=PC_{2^e}$ and \eqref{PP1} holds. In particular, $[P:P\cap C_{2^e}] = [PC_{2^e}:C_{2^e}]=2$.

Notice that the fact that $m\in\{-1,-1+2^{e-1}\}$ implies that the commutator $(C_{2^e} P)'$ has order $2^{e-1}\geq 4$.

Consider the group $PC_{n_2}$. It follows from $C_{2^e}\leq C_{n_2}$
that  $[P: P\cap C_{n_2}]\leq [P: P\cap C_{2^e}] = 2$. Together with $P\not\leq C_{n_2}$ we obtain that $[P: P\cap C_{n_2}] = [P: P\cap 
C_{2^e}] = 2$. Thus $|PC_{n_2}|=2n_2$. 

Since the orbits of $P C_{2^e}$ coincide with the orbits of $C_{2^e}\leq C_{n_2}$, the orbits of $PC_{n_2}$ coincide with the orbits of $C_{n_2}$. Therefore, $PC_{n_2} = C_{n_2} (PC_{n_2})_\omega$ for each $\omega\in\Omega$.
Counting orders we obtain $|(PC_{n_2})_\omega|=2$. The same counting for the group $PC_{2^e}$ yields $|(PC_{2^e})_\omega|=2$. Therefore
$(PC_{n_2})_\omega = (PC_{2^e})_\omega$. Let $b$ denote the unique involution in the group $(PC_{n_2})_\omega$. Then $x^b=x^\ell$ for every $x \in C_{n_2},$ where $\ell\in\{\pm 1,\pm 1 + n_2/2\}$. If $\ell\in\{1,1+n_2/2\}$, then $|(PC_{n_2})'|\leq 2$ contradicting 
 $|(PC_{2^e})'|\geq 4$.  Therefore 
$\ell\in\{ -1, -1+n_2/2\}$. In this case the group $PC_{n_2}\cong \Z_{n_2}\rtimes\sg{\ell}$. Therefore $PC_{n_2}\cong D_{2n_2},SD_{2n_2}$. In both cases $PC_{n_2}$ has two conjugacy classes outside $C_{n_2}$, say $Y_1$ and $Y_2$.
Since $P$ is normal in $PC_{n_2}$ and $P\not\leq C_{n_2}$, at least one of $Y_1$ and $Y_2$ is contained in $P$. If both are contained in $P$, then $P = PC_{n_2}$ or equivalently, $[P:C_{n_2}]=2$. This yields case (i). 

Assume now  that $P\neq PC_{n_2}$. Then $P$ contains just one class, say $Y_i$. Then $\sg{Y_i} \leq P < PC_{n_2}$. For any choice of $i,$  the subgroup $\sg{Y_i}$ has index two in $PC_{n_2}$.
Hence $P=\sg{Y_i}$. Analysing all possible cases we obtain part 
(ii) of our statement.
\end{proof}

\begin{thm}\label{Fit} With the notation of Theorem~\ref{main2},
either $F(G)=C,$ or $F(G) = \mathbf{O}_2(G)\times C_{2'}$ where $\mathbf{O}_2(G)$ is described in Lemma~\ref{L4}. 
\end{thm}
\begin{proof}

Recall that the Fitting subgroup $F(G)$ is a direct product of its subgroups ${\bf O}_p(G),$ where $p$ runs over the set of all prime divisors of $|G|$. 

It was shown before that  ${\bf O}_p(G) \leq C$ if $p > 2,$ implying that 
$F(G)$ is contained in $\mathbf{O}_2(G) \times C_{2'}$.  If $\mathbf{O}_2(G) \le C,$ then $F(G) \le C$. On the other hand, $C=C_G(C) \leq C_{G}(F(G)) \leq F(G),$ and we get $F(G)= C$.

Now, let $\mathbf{O}_2(G) \not\le C$. Then $\mathbf{O}_2(G)$ is given by parts 
(i)-(ii) of Lemma~\ref{L4}. 
We claim that $C_{2'}$ centralizes $\mathbf{O}_2(G)$. 
If $\mathbf{O}_2(G)\not\cong Q_8$, then 
$\aut(\mathbf{O}_2(G))$ is a $2$-group (see, e.\ g., 
\cite[Theorem~34.8(a) and (c)]{B08}), and the claim follows. Let now $\mathbf{O}_2(G)\cong Q_8$.  Pick an arbitrary $g \in C_{2'}$. Then $(\mathbf{O}_2(G)C_8)^g=(\mathbf{O}_2(G))^g(C_8)^g=\mathbf{O}_2(G)C_8,$ hence $g$ 
acts on $\mathbf{O}_2(G)C_8$ by conjugation as an element in $\aut (\mathbf{O}_2(G)C_8)$. 
As $\mathbf{O}_2(G)C_8 \cong SD_{16},$ $\aut (\mathbf{O}_2(G)C_8)$ is a $2$-group 
(see, e.\ g., \cite[Theorem~34.8(c)]{B08}), and we 
proved also in this case that any $g \in C_{2'}$ centralizes $\mathbf{O}_2(G)$. 
Then $C_{2'} \le C(F(G)) \le F(G),$ and so $F(G) = \mathbf{O}_2(G) \times C_{2'}$. 
\end{proof}

We are now ready to prove the main result of the paper.
\medskip

\noindent{\it Proof of Theorem~\ref{main2}.}
Recall that, $C$ is a cyclic subgroup of $\aut X$ and $G$ is a $\prec$-minimal subgroup of $\aut X$ which contains $C$.
By \cite[Theorem~1.8]{M99}, $G$ is solvable, so we can apply 
Theorem~\ref{Fit}. According to this theorem  either $F(G)=C,$ or 
$F(G) = \mathbf{O}_2(G)\times C_{2'}$ where $\mathbf{O}_2(G)$ is described in Lemma~\ref{L4}. If $F(G)=C,$ then $C$ is the only regular cyclic subgroup of $G$ implying $G=C$ (by $\prec$-minimality of $G$).

We are left with the case when $F(G)=\mathbf{O}_2(G)\times C_{2'},$ 
where $\mathbf{O}_2(G)$ is described in 
one of cases (i)-(ii) of Lemma~\ref{L4}.
If case (i) occurs, then $C$ is a unique cyclic subgroup of $F(G)$ of index two. Therefore $C$ is a unique regular cyclic subgroup of $G$ (if $\tilde{C}\leq G$ is another one, then by Lemma~\ref{L4}, $\tilde{C}\leq F(G)$).

Assume now that case (ii) of Lemma~\ref{L4} occurs. Then 
\begin{equation}\label{half}
 C_{n/2} \le F(G). 
\end{equation}
Assume first that $\mathbf{O}_2(G)\not\cong Q_8$. 
Then $C_{n_2/2}$ is a unique cyclic subgroup of $\mathbf{O}_2(G)$ of order $n_2/2$. So, in this case $C_{n/2}=C^2:=\{x^2 : x \in C\}$ is a unique cyclic subgroup of $F(G)$ of order $n/2$. Therefore, for any regular cyclic subgroup $\tilde{C}$ of $G$ it holds that $C^2 = \tilde{C}^2$. This implies that any regular cyclic subgroup of $G$ centralizes $C^2$. Take now  a subgroup $U\leq G$ generated by all regular cyclic subgroups of $G$. Then $U$ centralizes $C^2$ and $U\prec G$. By $\prec$-minimality of $G$ we obtain that $G=U$ which implies that $G\leq {\bf C}_{{\rm Sym}(\Omega)}(C^2)$.

The group $C^2$ is semiregular on both $\Omega$ and $\B$ and has two orbits on each of these sets. We denote these orbits by $\Omega_1,\Omega_2$ and $\B_1,\B_2$ respectively. Pick an arbitrary $\beta\in\B_1$. Since $k\geq 3$, the line $\beta$ is incident with at least two points $\alpha,\alpha'$ which are in the same $C^2$-orbit. 
Without loss of generality we may assume  that $\alpha,\alpha'\in\Omega_1$. Since $G$ centralizes $C^2$, the point stabilizer $G_\alpha$ fixes each point of $\Omega_1$. Hence $\alpha'^{G_\alpha}=\{ \alpha' \}$. Therefore any line $\beta'\in\beta^{G_\alpha}$ is incident with both $\alpha$ and $\alpha'$ too. Since $R$ is $K_{2,2}$-free, we obtain $\beta^{G_\alpha}=\{ \beta \}$. Therefore $G_\alpha \leq G_\beta$. Together with $|G_\alpha|=|G|/n =|G_\beta|$ we conclude that $G_\alpha = G_\beta$. Also, $G_\alpha$ fixes each element of $\B_1$ setwise, because $\B_1$ is a $C^2$-orbit of $\beta$.
If a line $\beta' \in \B_1$ is incident with only the elements of $\Omega_1$, then the lines of $\B_2$ are incident with the points of $\Omega_2$ only, contradicting connectedness of the incidence graph. 
Thus there is a point $\omega\in\Omega_2$ incident with $\beta$. Therefore the points of the set $\omega^{G_\beta}=\omega^{G_\alpha}$ are incident with
$\beta$ too.  The group $C^2$ acts transitively on $\Omega_2$ and commutes with $G_\alpha$ elementwise. Therefore, the orbits of $G_\alpha$ on $\Omega_2$ coincide with orbits of some subgroup $F\leq C^2$. Lemma~\ref{L1} implies that $F$ is trivial. But in this case $G_\alpha$ acts trivially on $\Omega$, and, therefore, is trivial. 
Hence $G=C$.

Consider now the remaining case $\mathbf{O}_2(G) \cong Q_8$. Since $F(G)_{2'}=C_{2'}$, it follows that all regular cyclic subgroups $\tilde{C}\leq G$ have the same odd part, that is, $C_{2'}=\tilde{C}_{2'}$ (see Lemma~\ref{L4}). Thus in order to prove that any two regular cyclic subgroups of $G$ are conjugate, it is
sufficient to show the conjugacy of their Sylow $2$-subgroups. Notice that, according to Lemma~\ref{L4} the Sylow $2$-subgroup of $C$ has order $8$.

Let $K\trianglelefteq G$ be the kernel of the action of $G$ on the imprimitivity system $\Omega/\mathbf{O}_2(G)$. The orbits of $\mathbf{O}_2(G)$ have the same length which should be equal to $8$ (otherwise the unique involution of $Q_8$ would act trivially). Since $C_8$ has the same orbits, we conclude $C_8\leq K$. 

Since $K$ and $\mathbf{O}_2(G)$ are normal in $G$, the 
centralizer $N=C_K(\mathbf{O}_2(G))$ is normal in $G$ too. Since $\mathbf{O}_2(G)$ acts regularly on each its orbit $\Delta\in\Omega/\mathbf{O}_2(G)$, the centralizer of $\mathbf{O}_2(G)^\Delta$ in ${\rm Sym}(\Delta)$ has order $8$. Therefore $N^\Delta$ is a $2$-group for any $\Delta\in\Omega/\mathbf{O}_2(G)$ implying that $N$ is a normal $2$-subgroup of $G$. Hence $N\leq Z(\mathbf{O}_2(G))$ implying $|N|=2$. Thus the group $K/N$ is embedded into $\aut(\mathbf{O}_2(G)) = \aut(Q_8)\cong S_4$.

As $\mathbf{O}_2(G) <  \mathbf{O}_2(G) C_8 \le K$ and 
$K$ is not a $2$-group, $|K|=48$. As $\mathbf{O}_2(G) C_8 \cong SD_{16},$ different cyclic 
subgroups of order $8$ are contained in different Sylow $2$-subgroups of $K$.  
Since all Sylow $2$-subgroups are conjugate and contain a unique cyclic subgroup of order $8$, all cyclic subgroups of order $8$ are conjugate in $K$.
 
\hfill $\square$ 

We would like to end the paper with the following problem.

\begin{prob}
Classify all finite groups which are CI-groups for balanced 
configurations.
\end{prob}

\end{document}